\documentclass[12pt]{article}
\usepackage{color,latexsym,amsfonts,amssymb}
\usepackage{hyperref}
\usepackage{amsmath,cite}

\allowdisplaybreaks[4]
\usepackage{amsthm}
\usepackage[mathscr]{euscript}
\newcommand{\scrB}{{\mathscr{B}}}
\newcommand{\E}{\mathbb{E}} 

\newcommand{\R}{\mathbb{R}}

\def\pr{\mathbb{P}}

\def\equald{{\stackrel{d}{=}}}

\newcommand{\mean}{{\E}}
\def\prob{{\pr}}
\def\var{{\rm Var}}

\def\pR{{\R_+^\circ}}%positive real numbers 

\newcommand{\convd}{{\stackrel{d}{\longrightarrow}}}

\newcommand{\ca}{{\cal A}}
\newcommand{\cb}{{\cal B}}
\newcommand{\cd}{{\cal D}}
\newcommand{\cf}{{\cal F}}

\newcommand{\ch}{{\cal H}}

\newcommand{\ck}{{\cal K}}
\newcommand{\cl}{{\cal L}}

\newcommand{\cs}{{\cal S}}

\def\ignore#1{}

\def\convp{\stackrel{{{\prob}}}{\longrightarrow}}

\def\Ref#1{(\ref{#1})}

\def\MXs{M^\Xi_\sigma}
\def\MLs{M^\Lambda_\sigma}

\def\pxs{p^{(x)}_\sigma}
\def\1{{\bf 1}}
\def\Var{{\rm Var}}

\newcommand{\beas}{\begin{eqnarray*}}
\newcommand{\enas}{\end{eqnarray*}}

\newcommand{\bea}{\begin{eqnarray}}
\newcommand{\ena}{\end{eqnarray}}

\def\eq{\begin{equation}}
\def\en{\end{equation}}

\newtheorem{Theorem}{Theorem}[section]
\newtheorem{Corollary}[Theorem]{Corollary}
\newtheorem{Lemma}[Theorem]{Lemma}
\newtheorem{Proposition}[Theorem]{Proposition}

\theoremstyle{definition}
\newtheorem{Definition}[Theorem]{Definition}

\newtheorem{Remark}[Theorem]{Remark}
\newtheorem{Example}[Theorem]{Example}

\newtheorem{Setup}[Theorem]{Setup}

\makeatletter
\renewcommand\theequation{\thesection.\@arabic\c@equation}
\makeatother

\def\IR{\R}
\def\IN{\mathbb{N}}
\def\abs#1{|#1|}
\def\be#1{\begin{equation*}#1\end{equation*}}
\def\ben#1{\begin{equation}#1\end{equation}}
\def\eq#1{\eqref{#1}}
\def\ba#1{\begin{align*}#1\end{align*}}
\def\ban#1{\begin{align}#1\end{align}}

\def\eps{\varepsilon}

\def\mcI{\mathcal{I}}
\def\dtv{d_{\textrm{TV}}}
\def\law{\mathcal{L}}

\def\t#1{^{(#1)}}
\def\res#1{|_{#1}}

\def\tt{\tau}
\def\C{\mathbb{C}}

\begin{document}

\title{%Poisson process approximation for the signal strength in wireless networks\\
%Approximating wireless network signals with Poisson processes\\
%Wireless network signals appear Poisson\\
%When wireless network signals appear Poisson\\
When do wireless network signals appear Poisson?
}

\author{H. Paul Keeler\thanks{Weierstrass Institute, Berlin, Germany; keeler@wias-berlin.de; 
work supported in part by Australian Research Council Discovery Grant DP110101663},  
Nathan Ross\thanks{Department of Mathematics and Statistics, the University of Melbourne, 
Parkville, VIC 3010, Australia; nathan.ross@unimelb.edu.au}
and Aihua Xia\thanks{Department of Mathematics and Statistics, the University of Melbourne, 
Parkville, VIC 3010, Australia; aihuaxia@unimelb.edu.au;
work supported in part by Australian Research Council Discovery Grant DP120102398}}

%\thanks{University of Melbourne and Weierstrass Institute for Applied Analysis and Stochastics, keeler@wias-berlin.de; work supported in part by Australian Research Council Grants No DP110101663}

\date{\today}

\maketitle

\abstract{We consider the point process of signal strengths from transmitters in a wireless network observed from a fixed position under models with general
signal path loss and random propagation effects. We show via coupling arguments that under general conditions this point process of signal strengths can be 
well-approximated by an inhomogeneous Poisson or a Cox point processes on the positive real line. 
We also provide some bounds on the total variation distance between the laws of these point processes and both Poisson and Cox point processes.
Under appropriate conditions, these results support the use of a
spatial Poisson point process for the underlying positioning of transmitters in models of wireless networks, even if in reality the positioning does not appear Poisson. 
We apply the results to a number of models with popular choices for positioning of transmitters, path loss functions, and distributions of propagation effects.  
}

\section{Introduction}

In this article we study signal strengths in stochastic models of wireless networks such as ad hoc, sensor, and mobile or cellular phone networks. The building blocks of the models considered are point processes (from now on, we simply refer to them as processes) and stochastic geometry; 
for background see\cite{baccelli2009stochastic1,baccelli2009stochastic2,haenggi2009stochastic,haenggi2012stochastic}. 
Details of models can differ depending on the type of network, but the standard framework assumed throughout the article is that 
an observer is placed at the origin of $\IR^d$ and transmitters are located at positions $\xi=\{x_i: \ i\in \IN:=\{1,2,\dots\}\}\subseteq \IR^d/\{0\}$
(equivalently we write $\xi=\sum_{i\in\IN}\delta_{x_i}$, where $\delta_x$ is the Dirac measure at $x$) either deterministically  or according to a random process.
In the absence of  \emph{propagation effects},
the signal received by the observer from a transmitter located at $x\in\IR^d/\{0\}$ has strength given by a deterministic \emph{path loss} function $\ell(x)$, typically taken to be a function of $\abs{x}$, the distance to the origin; a standard 
assumption is $\ell(x)= C \abs{x}^{-\beta}$ for some $\beta>0, C>0$. The random propagation effects are assumed to influence the strength of the signal via, for example, \emph{multipath fading} (due to signals taking multiple paths and colliding with each other) and \emph{shadow fading} or \emph{shadowing} (due to signals colliding with large obstacles such as buildings). {Although such  effects may occur on different scales,} we use the general term ``fading" to refer to {all} types of random propagation effects, which are incorporated into the model via a sequence of i.i.d.\ positive {random} variables $S, S_1, S_2,\ldots$, where the signal power or strength from transmitter $x_i$ is given by
\begin{equation}
P_i=\ell(x_i) S_i=:\frac{S_i}{g(x_i)};
\end{equation}
here and below $g(x):=1/\ell(x)$.
Understanding the distribution of the process 
\begin{equation}
\Pi:=\{P_i\}_{i\in\IN}, 
\end{equation}
on ${\pR:=(0,\infty)}$
 and various functions of it (for example, the signal-to-interference ratio $\{P_i/(\sum_{j} P_j - P_i)\}_{i\in \IN}$, or the {largest signal strength} $\max_i\{P_i\}$)
under different assumptions on~$S$, $g$ and $\xi$ is a major goal of wireless network modeling.

A common assumption is that the transmitter positions $\xi$ are given by a homogeneous Poisson process. In this case easy theory 
implies that $\Pi$ is also a {(typically inhomogeneous)} Poisson process (easy theory is the Poisson mapping theorem: the pairs $\{(x_i, S_i)\}_{x_i\in\xi}$ form a  Poisson process
and the points of $\Pi$ are a measurable function of these points). The only potential difficulty with this framework is with computing the mean
measure of $\Pi$ for explicit choices of $S$ and $g$. The assumption that the transmitters follow a Poisson process is usually justified by  
thinking of the observer as an ``average" observer {with 
fixed transmitters}, though empirical studies are less clear on the matter~\cite{lee2013stochastic}.
Our main purpose is to investigate
the behavior of $\Pi$ when the Poisson process transmitter assumption is relaxed and in particular to answer the titular question
of the present article.

A first step is taken in \cite{blaszczyszyn2013using,blaszczyszyn2014wireless} where it is shown that if ${S=}S(\sigma) = \exp\{\sigma B-\sigma^2/\beta\}$
 for $B$ a standard normal {random} variable and $\beta>2$, $g(x)=(K \abs{x})^\beta$ for a constant $K>0$, and $\xi$ is such that {$|\xi|(r)$, the number of points within
 distance $r$ of the origin, satisfies} ${|\xi|(r)/(\pi r^2)}\to \lambda>0$ (a.s.\ if $\xi$ is random), then as $\sigma\to\infty$, $\Pi$ converges to a Poisson process.
One of our main theorems is to greatly generalize this result, providing simple criteria for Poisson convergence.

A critical point before stating the theorem: to match empirical observation, the signal power process $\Pi$ 
should have a large number of very weak signals, which means any approximating Poisson process should have a mean measure with a (near) singularity at zero. 
To focus on the non-singular part, we instead study the process on ${\pR}$ 
of the inverses of the signal power values, 
\[
N=\{P_i^{-1}\}_{i\in \IN}=\left\{\frac{g(x_i)}{S_i}\right\}_{i\in\IN},
\]
though results about $N$ typically can be translated to results
for $\Pi$ by inverting the points and using the Poisson mapping theorem; see, for example, Remark~\ref{remN2P} below.  
The process $N$ on ${\pR}$ is referred to as {\em propagation (loss) process}~\cite{blaszczyszyn2010impact} (or {\em path loss with fading process}~\cite{haenggi2008geometric}) {generated by $S$, $g$ and $\xi$}. 
We have the following result; denote convergence in probability by~$\convp$.

\begin{Theorem}\label{thmshadow}
Let $\xi{\subset\IR^d/\{0\}}$ be a locally finite set of points such that there is a nondecreasing function $D$ satisfying (almost surely if $\xi$ is random)
\ben{
\lim_{r\to\infty} \frac{{|\xi|}(r)}{D(r)}=1 \,\,\,\,\mbox{ and } \,\,\,\,\, \lim_{r\to0} D(r)=0, \label{ee1}
}
and let {$g:\IR^d\to\IR_+:=[0,\infty)$
such that $g$ is positive on $\IR^d/\{0\}$ and $g(x)=h(\abs{x})$ for $h$ a left continuous and nondecreasing
 function with inverse $h^{-1}(y):=\inf\{x: h(x)> y\}$.}
Let $(S(\sigma))_{\sigma\geq0}$ be a family of positive random variables indexed by some non-negative parameter~$\sigma$, $N\t \sigma$ be the propagation 
process generated by $S(\sigma)$, $g$ and $\xi$.
If
\be{
(i)\ S(\sigma)\convp0\,\,\,\, \mbox{ and } \,\,\,  (ii)\ \mean D( h^{-1}(S(\sigma) t)) \to L(t),\ {\forall t\in C(L),}
}
as $\sigma\to\infty$, {where $C(L):=\{t\in\pR:\ \lim_{s\to t} L(s)=L(t)\}$,} then $N\t \sigma$ converges weakly to a Poisson process {on $\pR$} with mean measure $L$.
\end{Theorem}

\begin{Remark}\label{rem:logno}
The result of \cite[Theorem 3]{blaszczyszyn2013using} already mentioned is an easy consequence of Theorem~\ref{thmshadow}.
Their assumptions are $d=2$, $S(\sigma) = \exp\{\sigma B-\sigma^2/\beta\}$
 for $B$ a standard normal {random} variable and $\beta>2$, $g(x)=(K \abs{x})^\beta$ for a constant $K>0$, and $\xi$ is such that 
 ${|\xi|}(r)/(\pi r^2)\to \lambda>0$. Thus we set $D(r)=\lambda \pi r^2$ and
 then noting that $\mean S(\sigma)^{2/\beta}=1$,
  Theorem~\ref{thmshadow} implies $N\t \sigma$ converges
 weakly to a Poisson process with mean measure
 \[
{L}(t)=\frac{\lambda \pi t^{2/\beta}}{K^2}.
 \]
 Notice that with the same choices of $g$ and $\xi$, 
 \emph{any} $S(\sigma)$ such that  $S(\sigma)\convp0$ and $\mean S(\sigma)^{2/\beta} \to 1$ has the same 
 Poisson process limit as that of the sequence of lognormal {random} variables.
 
 In particular, the same result holds for \emph{composite} fading models
\cite{reig2013estimation} of product type having $S(\sigma)=S_L(\sigma) S_F$ with
$S_L(\sigma)\convp0$ and $\mean S_L(\sigma)^{2/\beta} \to 1$ (for example the lognormal {random} variable above)
 and the random variable $S_F$ is independent of $S_L(\sigma)$ %, where $S_F$  is \emph{any} distribution satisfying  
with $\mean S_F^{2/\beta}=1$. If $S_F$ 
 is an exponential random variable with rate $\Gamma(1+2/\beta)^{\beta/2}$, and $S_L(\sigma)$ is the lognormal {random} variable above,
 then $S(\sigma)$ is known as a Suzuki model~\cite{reig2013estimation}. 
 Actually, all that is required for convergence is $\mean S_F^{2/\beta}<\infty$ in which case $N\t \sigma$ converges to a Poisson process
 with mean measure ${L}(t)=\lambda  \mean S_F^{2/\beta} \pi t^{2/\beta}/K^2$, c.f., \cite[Corollary 12]{blaszczyszyn2014wireless}.

Moreover, by applying the Poisson mapping theorem,
it's
easy to see that the limiting Poisson process can also be realized by applying the function $g$ to the points of a homogenous Poisson process on $\IR^2$ with intensity $\lambda$.
Thus
for modeling purposes, {in many situations where} Theorem~\ref{thmshadow} roughly applies (that is, the fading variables are small with sufficiently large probability) then one may assume that the underlying transmitter
configuration is generated by a Poisson process.
\end{Remark}
  
\begin{Remark} 
Intuitively, because $S(\sigma)$ tends to zero, most points of $\xi$ are being sent out to infinity in $N\t \sigma$. Large values of $S(\sigma)$ transform the  far away points of $\xi$
closer to the origin, so a non-degenerate limit can only occur when, as $\sigma\to\infty$, $S(\sigma)$ becomes large on a set with probability 
inversely proportional to the number of points of $\xi$ at the appropriate distance.
In particular, it's necessary that $\lim_{r\to\infty}D(r)=\infty$ to have a non-degenerate limit process  (that is, with $L(t)$ not identically zero).
\end{Remark}

{\begin{Remark} In essence, the Poisson limit is due to the thinning of the points in $\xi$. However, 
as the retained points are redistributed, the thinning scheme
in this paper is very different from the classical thinning schemes in the literature{, e.g.,}
\cite{Kallenberg1975, Brown1979, Schuhmacher2005, Schuhmacher2009}.
\end{Remark}}

\begin{Remark}
The condition~\eq{ee1} on $\xi$ is satisfied for transmitters placed on a regular lattice excluding the origin or 
that are a realization of a stationary, ergodic process with no points at the origin; in the former case 
we would take $D{(r)}$ proportional to the volume of a ball of radius $r$ in $d$ dimensions. 
\end{Remark}

Theorem~\ref{thmshadow} is one of a number of approximation and convergence results we establish in this paper that apply to the inverse signal strength process $N$; we state these in detail in Section~\ref{seclim}. 
In Section~\ref{secex} we provide an overview of some standard wireless network models and apply our results to these. Section~\ref{secproofs} contains proofs
and we conclude the paper with some discussion in Section~\ref{14102301}.

\section{Poisson process approximating signal powers}\label{seclim}
%To match empirical observation, the signal power process $\Pi$ 
%should have a large number of very weak signals, which means any approximating Poisson process should have a (near) singularity at zero. To focus on the non-singular part, we instead study the process on $\pR$ of the inverses of the signal strengths, $\{P_i^{-1}\}_{i\in \IN}$, though our results below also hold for $\Pi$ by inverting the points and using the Poisson mapping theorem; see, for example, Remark~\ref{remN2P}.  
%The transformation $\{P_i^{-1}\}_{i\in \IN}$ is referred to as the {\em propagation (loss) process}~\cite{blaszczyszyn2010impact} (also {\em path loss with fading process}~\cite{haenggi2008geometric}). 

We state our formal setup.

\begin{Setup}[Main Setup]\label{setup}
Let $\xi{\subset\IR^d/\{0\}}$
be a locally finite  
collection of points in 
$\IR^d$ % $(\IR^d,\cf(\IR^d))$ % (to match our setup above, take $\cs=\IR^d/\{0\}$)
 and $g$ 
be a measurable mapping from $(\IR^d,\scrB(\IR^d))$ to {$({\IR_+},\scrB({\IR_+}))$} such that $\{g(x): x\in\xi\}$ is locally finite; here %$\real_+=(0,\infty)$ (so $g(x)\not=0$) and 
$\scrB(\cdot)$ denotes the Borel $\sigma$-algebra. % on $\real_+:=(0,\infty)$ 
Write $\xi=\{x_i: \ i\in \mcI^{{\xi}} \}$ where $\mcI^\xi$ is a finite or countable index set
(without loss, taken to be $\{1,\ldots,n\}$ in the finite case and $\IN$ in the infinite case){. Define $|\xi|(r)$ as the number of points of $\xi$ within distance $r$ of the origin}
and let $\{S, S_{i}:\ i\in \mcI^{{\xi}} \}$ be a sequence of independent {and identically distributed} \emph{positive} random variables. 
Set $Y_{i}=g(x_{i})/S_{i}$ and let $N$ be the propagation process generated by the collection
$\{Y_i\}_{i\in\mcI^{{\xi}}}$, that is, $N=\sum_{i\in \mcI^{{\xi}}}\delta_{Y_{i}}$.
%, where
%$\delta_t$ is the Dirac measure at $t$. 
Let {$p^{(x)}(t)=\pr({0<}g(x)/S\le t)$, $x\in \IR^d$,} $p_i(t)={p^{(x_i)}(t)}$, $M(t):=M^{{\xi}}(t):=\sum_{i\in\mcI^{{\xi}}} p_i(t)$ and $Z$ be a Poisson process on ${\pR}$ having mean measure
$M(t)$. (Since $M$ is non-decreasing and non-negative, $Z$ is well-defined.) For any Radon measure $\eta$ on $\pR$ and {$\tau>0$}, we define {$\eta(\tt):=\eta((0,\tt])$ and} $\eta\res\tt$ as 
the Radon measure restricted to the interval $(0,\tt]$.
\end{Setup}

Recall the total variation distance between two probability measures $\nu_1, \nu_2$ on the same measurable space $(\cd, \cf(\cd))$ is defined as
\be{
\dtv(\nu_1, \nu_2)=\sup_{A\in\cb(\cd)}\abs{\nu_1(A)-\nu_2(A)}.
}
Total variation distance is a common and strong metric that bounds the maximum difference in probabilities between two probability distributions. 
In our context, we take $\cd=\ch$, the space of all locally finite point measures on $\pR$ or $\IR^d$ equipped with the vague topology \cite[p.~169]{Kallenberg1983}, and $\cf(\cd)=\cb(\ch)$, the Borel $\sigma$-algebra generated by the vague topology. 

We now present a bound on the total variation distance between {the distributions} of processes {$N\res\tt$ and $Z\res\tt$},
% restricted to the interval $(0,\tt]$, 
where $1/\tt$ can be interpreted as the smallest possible power value of interest for an observer in the network.
Here and below we us $\law(X)$ to denote the distribution of a random element~$X$.

\begin{Theorem}\label{thmdtv}
Assuming the main Setup~\ref{setup}, 
{we have}
\be{
\frac{ 1\wedge M(\tt)^{-1} }{32}  \sum_{i\in\mcI^\xi} p_i(\tt)^2\leq \dtv(\law(Z\res \tt), \law(N\res \tt))\leq \sum_{i\in\mcI^\xi} p_i(\tt)^2\leq M(\tt) \sup_{i\in \mcI^\xi} p_i(\tt).
}
\end{Theorem}

%\cyan{\begin{Remark}
%In the definition of total variation distance specialized to Theorem~\ref{thmdtv}, our underlying probability space is $D_{\IR_+}$, the collection of all functions on $\IR_+$ which are right continuous and with left limits,
%equipped with the Skorokhod topology~\cite[Chapter~3]{Billingsley1999} so that $\cg$ is the $\sigma$-algebra generated by the Skorokhod
%topology. A random process on $\IR_+$ is then defined as a measurable mapping from an underlying probability space to $(D_{\IR_+},\cg)$.
%\end{Remark}}

%\begin{Remark}
%The bound is tight in the following sense. If $N(\tt)$ is the number of points of $N$ in $(0,\tt]$ and
%$N$ is close to a Poisson process, then $N(\tt)$ must be close to a Poisson distribution. But
%for $N(\tt)$ to be approximately Poisson, 
%the mean and variance must be close and
%\be{
%\var(N(\tt))=\mean N(\tt) - \sum_{i\in\mcI^\xi} p_i(\tt)^2,
%}
%so the bound of the theorem must be small if the $Z\res \tt$ is to be close to $N\res \tt$.
%The upper bound is on the order of what would be expected if we were only bounding the total variation distance between the \emph{number} of points
%of $N$ falling in the interval $(0,T)$ and the appropriate Poisson variable; see \cite{lecam1960, BarbourHolstJanson1992}. 
%Typically obtaining such a bound for the Poisson process is
%difficult \cite{Aihua's stuff}. 
%\end{Remark}

\begin{Remark}\label{remN2P} Since the total variation distance is preserved under one to one mappings, Theorem~\ref{thmdtv} also implies 
that for $Z'$ the Poisson process with mean measure $M'{[}\tt, \infty)=M(1/\tt)$, and denoting the restriction of $Z'$ to $[\tt, \infty)$ by $Z'|^{\tt}$, 
\be{
\dtv(\law(Z'|^{\tt}), \law(\Pi|^{\tt}))\leq \sum_{i\in\mcI^\xi} p_i(1/\tt)^2\leq  M(1/\tt) \sup_{i\in \mcI^\xi} p_i(1/\tt),
}
where $\Pi$ is the process of signal powers defined in the introduction.
\end{Remark}

\begin{Remark}
Theorem{~\ref{thmdtv}} shows that the main criterion for Poisson process convergence in the wireless network setting is that,
in the notation of the theorem, 
{$\sup_{i\in \mcI^\xi} p_i(\tt)$ is small.}
The other criterion is the convergence of the mean measure, though the theorem still technically applies even if $M(t)$ is infinite or zero.
\end{Remark}

Theorem~\ref{thmdtv} is an approximation theorem, but we can use it to prove convergence results in terms of the vague topology, denoted by $\convd$;
note that such convergence implies convergence of natural statistics of the process, e.g., point counts for finite collections of {relatively compact} Borel sets
    and random variables that are integrals against the point measure of continuous and compactly supported functions $f: \pR\to \R_+$
\cite[Theorem~4.2]{Kallenberg1983}.
We have the following corollary of Theorem~\ref{thmdtv}.

\begin{Corollary} \label{corollary1}
For each $n$, let $\xi_n=\{x_{ni}\}_{i\in\mcI_n}$, $\mcI_n:=\mcI^{\xi_n}$, $S_{ni}$ and $Y_{ni}$ satisfy the main Setup~\ref{setup},
and define $N\t n$ as the process generated by the $\{Y_{ni}\}_{i\in \mcI_n}$. 
If for all {$t\in C(L)$},  
\be{
\lim_{n\to\infty}\sup_{i\in\mcI_n} \pr({0<}Y_{ni}\leq t)= 0
\mbox{ and }
\lim_{n\to\infty}\mean N\t n(t)= L(t),}
%\label{X01}
%\sum_{i\in\mcI_n} \pr(Y_{ni}\leq t)^2\to 0\label{X01}
%and $M\t n(t):=\mean N\t n([0,t])\to M(t)$ , 
then $N\t n \convd Z^{{L}}$, %as $n\to\infty$, 
where $Z^{{L}}$ is a Poisson process on $\pR$ with mean measure
$L$.
\end{Corollary}

%\begin{Remark}\label{remcor1}
%The conclusion of the theorem holds assuming $\sup_{i\in\mcI_n} \pr(Y_{ni}\leq t)\to 0$ rather than~\eq{X01} since
%\be{
%\sum_{i\in\mcI_n} \pr(Y_{ni}\leq t)^2\leq M\t n(t) \sup_{i\in\mcI_n} \pr(Y_{ni}\leq t),
%}
%and we're assuming $M\t n(t)\to M(t)$.
%\end{Remark}

%
%\begin{proof}
%For each $t\in\real_+$, $\sup_{i\in I_n}\pr(Y_{ni}\le t)=\sup_{i\in I_n}\pr(S_{ni}\ge g(x_{ni})/t)\to 0$ as $n\to \infty$, hence $N_n$ can be viewed as the superposition of independent sparse processes and the claim follows from Grigelionis~(1963).
%\end{proof}
%
% 
 The caveat to Theorem~\ref{thmdtv} and Corollary~\ref{corollary1} is that
 the mean measure of the processes may be difficult to compute in practice. The 
 following
 proposition provides alternative formulas
 for the mean measure. {Recall that for a non-decreasing function $h$
 we define the inverse $h^{-1}(y)=\inf\{x:h(x)>y\}$.}
 
\begin{Proposition}\label{propmean}
Recall the notation of the main Setup~\ref{setup} and assume %$\cs=\IR^d$, 
%${\xi=\sum_{i\in\mcI^{{\xi}}}\delta_{x_i}}$ is a locally finite point measure {on $\cs$} and 
$g(x)=h(\abs{x})$ for some function~$h$. 
%Let ${\xi}(r)$ be the number of points of $
%{\xi}$ within distance $r$ of the origin.
Then 
\ben{
M(t):=M^{{\xi}}(t):=\sum_{i\in\mcI^{{\xi}}} p_i(t) =\int_{0}^\infty \pr\left({0<}\frac{h(r)}{S}\leq t\right) {|\xi|}(dr). \label{meanform}
} 
If in addition, $h$ is {positive on $\pR$}, {left continuous and nondecreasing with inverse $h^{-1}$}, then
\ben{\label{meanform2}
M(t)=\mean \left[{|\xi|}(h^{-1}(St))\right].
}
\end{Proposition}

%\begin{Theorem}\label{thmshadow}
%Let $(S(\sigma))_{\sigma\geq0}$ be a family of positive random variables indexed by some non-negative parameter~$\sigma$, $N\t \sigma$ be the propagation 
%process generated by $S(\sigma)$ and $\xi$, and 
%let $Z\t \sigma$ be the Poisson process with
%mean measure $M\t \sigma=\mean\left[  \xi\left( h^{-1}\left(S(\sigma) t\right)\right)\right]$.
%If $x_1$ is the point of $\xi$ closest to the origin,
%then 
%\be{
%\dtv(\law(N\t \sigma|_{\tt}), \law(Z\t \sigma|_{\tt})) \leq M\t \sigma(t) \pr\left(S(\sigma)\geq \frac{g(x_1)}{\tt}\right).
%}
%Further, if there is a positive non-decreasing function $D$ such that
%\be{
%\lim_{r\to\infty} \frac{\xi(r)}{D(r)}=1 \,\,\,\,\mbox{ and } \,\,\,\,\, \lim_{r\to0} D(r)=0.
%}
%%for any $\delta>0$,
%\be{
%(i)\ S(\sigma)\convp0\,\,\,\, \mbox{ and } \,\,\,  (ii)\ \mean D( h^{-1}(S(\sigma) t)) \to M(t),
%}
%as $\sigma\to\infty$, then $N\t \sigma$ converges weakly to a Poisson process with mean measure $M(t)$.
%\end{Theorem}

\subsection{Random transmitter positions}

If the transmitters are placed according to a random process {$\Xi$} independent of the fading sequence,
then Theorem~\ref{thmdtv} holds conditionally, but the mean measure 
of the approximating Poisson process may change for different realizations of $\Xi$
and so in general we approximate $N$ by a Cox process, which is a Poisson process with a random intensity measure. Note that even in this case, each user sees a Poisson process, but the intensity measures of users in different locations and at different times may be different.

We change the main Setup~\ref{setup} by replacing $\xi$ with a process $\Xi$. {According to \cite[Lemma~2.3]{Kallenberg1983}, we can write 
$\Xi=\sum_{i\in \mcI^\Xi}\delta_{X_i}$
%\note{@Aihua: maybe better to write $\Xi=\sum_{i\in\mcI^\Xi}\delta_{X_i}$ or $\Xi=\sum_{1\le i\le |\mcI^\Xi|}\delta_{X_i}$?}
 with (possibly random) index set $\mcI^\Xi$.}  %; that is $\Xi=\sum_{1\le i\le \mcI^\Xi}\delta_{X_i}$. 
Again, we let   
$\{S,S_i:\ i\in\IN\}$ be a sequence of positive i.i.d.\ random variables
independent of $\Xi$ 
and we 
define $N:=N^\Xi$ as in the main Setup~\ref{setup} {but with $\Xi$ replacing $\xi$}.

%We change the main Setup~\ref{setup} so the (finite or countable) index set
%$\mcI$ is now random, though again we take it to be $\{1,\ldots, n\}$ in the finite case 
%or the positive integers in the infinite case. Formally, we let $(S_1,S_2,\ldots)$ be a sequence of i.i.d.\ variables
%independent of the random set $\Xi=\{x_1,x_2,\ldots\}$ (labeled arbitrarily) indexed by $\mcI^\Xi$ 
%and we 
%define $N$ conditional on $\Xi$ as in the main Setup~\ref{setup}.

Before stating our Cox process approximation result, we cover the important case 
where $\Xi$ is approximately a Poisson process, in which case Poisson process approximation is valid.
\begin{Theorem}\label{thmPPP}
%Recall the main Setup~\ref{setup}, but now let $\Xi$ be random and independent of $\{S_i\}_{i \in \IN}$.
With the setup above,
let  $\Theta$ be a Poisson process on $\IR^d$ 
%$\cs$
independent of $\{S_i:i \in \IN\}$ and define 
$Z=\sum_{{\theta_i}\in\Theta} \delta_{g({\theta_i})/S_i}$,
%$Z=\{g(\theta_i)/S_i: \theta_i\in \Theta\}$, 
{as} a process on ${\pR}$.
% with mean measure $M(t)= \mean \sum_{i\in\mcI^\Xi} P(Y_i\leq t | \Xi)$.
Then
$Z$ is a Poisson process with mean measure 
\[
M(t)=\E Z(t)=\mean \int \pr({0<}g(\theta)/S\leq t)\Theta(d\theta)
\]
 and \be{
\dtv(\law(Z\res \tt), \law(N\res \tt))\leq \dtv(\law(\Theta), \law(\Xi)).
}
\end{Theorem}

Theorem~\ref{thmPPP} shows that if $\Xi$ is close to \emph{some} Poisson process, then $N$ is close to a Poisson process.
In general, random $\Xi$ give rise to Cox processes which are Poisson processes with random mean measures. More precisely,
we say that $Z$ is a Cox process directed by the random measure $M$ if 
conditional on $M$, $Z$ is a Poisson process with mean measure $M$. 
We have the following result.
%To state our next result, let
%\be{
%M^\Xi(t)= \sum_{i\in\mcI^\Xi} \pr(Y_i\leq t  | \Xi ).
%}

\begin{Theorem}\label{thmdtv2}
Recall the main Setup~\ref{setup}, but with $\xi$ replaced by a locally finite process $\Xi$ independent of $\{S_i\}_{i \in \IN}$. 
Define
\be{
%p_i^\Xi (t)=\pr(Y_i\leq t | \Xi), \,\,\,\,\,\, M^\Xi(t)=\sum_{i\in \mcI^{\Xi}} p_i^\Xi(t).
{M^\Xi(t)=\int_{\IR^d}p^{(x)}(t)\Xi(dx).}
}
Let $Z$ be the Cox process 
directed by the measure $M^\Xi$, that is, conditional on $\Xi$, $Z$ is a Poisson process with mean measure $M^\Xi$.
Then
\be{ %\label{ranpnbd}
\dtv(\law(Z\res \tt), \law(N\res \tt))\leq  \mean {\int_{\IR^d}p^{(x)}({\tau})^2\Xi(dx)}
%\sum_{i\in\mcI^\Xi} p_i^\Xi(\tt)^2.
%\leq M(\tt) \inf_{i\in \mcI} p_i(\tt).
}
\end{Theorem}
The following conditional and unconditional analog of Proposition~\ref{propmean} holds for random $\xi$; we omit the proof because it's straightforward from Proposition~\ref{propmean}.
\begin{Proposition}\label{propmean2}
Recall the notation of the main Setup~\ref{setup} with $\xi$ replaced by a locally finite process $\Xi$ independent of $S_1,S_2,\ldots$ 
and assume $g(x)=h(\abs{x})$ for some function $h$. Then
%Let $\Xi(r)$ be the number of points of $\Xi$ within distance $r$ of the origin, then 
\ben{
M^\Xi(t):={\int_{\IR^d}p^{(x)}(t)\Xi(dx)}%\sum_{i\in\mcI^\Xi} \pr(Y_i\leq t \big| \Xi)
 =\int_{0}^\infty \pr\left({0<}\frac{h(r)}{S}\leq t\right) {|\Xi|}(dr) \label{meanforma}
} 
and 
\be{
M(t):={\int_{\IR^d}p^{(x)}(t)\Lambda(dx)}%\sum_{i\in\mcI^\Xi} \pr(Y_i\leq t) 
=\mean \int_{0}^\infty \pr\left({0<}\frac{h(r)}{S}\leq t\right) {|\Xi|}(dr), \label{meanformb}
}
{where $\Lambda$ is the mean measure of $\Xi$.}
If in addition, $h$ is {positive on $\pR$,} {left continuous and nondecreasing with inverse $h^{-1}$}, then
\ben{\label{meanform2a}
M^\Xi(t)=\mean \left[{|\Xi|}(h^{-1}(St)) \big| \Xi \right]
}
and
\ben{\label{meanform2b}
M(t)=\mean \left[{|\Xi|}(h^{-1}(St)) \right]{=\mean \left[|\Lambda|(h^{-1}(St))\right]},
} 
{where $|\Lambda|(r):=\mean |\Xi|(r)$.}
\end{Proposition}

\subsection{Poisson versus Cox} 

When the transmitters are randomly placed according to a process $\Xi$, it's possible for the propagation process $N$ 
to be close to a Poisson process (Theorem~\ref{thmPPP}) or Cox process with non-deterministic {mean} measure (Theorem~\ref{thmdtv2}). 
There are realistic situations where we expect {the propagation process to be close to} a Cox process, for example, if the fading distribution $S$ is a mixture of distributions 
(see Section~\ref{sec:compo} for an example) or the observer initially
connects in a random way to one of a number of different networks that have different transmitter coverage densities or fading distributions
(see the example (iii) below for a toy model). However, it's of interest to understand when 
a Poisson, rather than Cox, process is appropriate. Theorem~\ref{thmdtv2} suggests that if $M^{\Xi}(t)$ is
close to deterministic, then $N$ may be close to Poisson. 
The next result is a convergence version of this statement with easily checkable conditions.

\begin{Theorem} \label{thm2} 
Assume that $\Xi$ is a process on $\IR^d$ with a locally finite mean measure $\Lambda$ such that 
%for ${|\Lambda|}(r):=\E{|\Xi|}(r)$, 
$\lim_{r\downarrow 0}{|\Lambda|}(r)=0$ and as $r\to\infty$,
\ben{
{|\Lambda|}(r)\to \infty, \,\,\,\,\,\, \Var({|\Xi|}(r))/({|\Lambda|}(r))^2\to0.\label{2014110901}
}
Let $(S(\sigma))_{\sigma\geq0}$ be a family of positive random variables, $N\t \sigma$ be the propagation process generated by $S(\sigma),$ $g$ and $\Xi$.
Assume $g(x)=h(|x|)$, where $h$ is {left continuous, nondecreasing and positive} on $\pR$.
If
\be{
(i)\  S(\sigma) \convp 0  \,\,\,\,\,\, \mbox{ and } \,\,\,\,\,\,  (ii)\  \int_{\IR^d} \pr\left({0<}\frac{g(x)}{S(\sigma)}\le t\right)\Lambda(dx)\to L(t)\mbox{ for all }t\in C(L)%\mean \Lambda( h^{-1}(S(\sigma) t)) \to L(t) 
}
as $\sigma\to\infty$,
then
$N\t \sigma$ converges weakly to a Poisson process {$Z^L$} with mean measure $L$.
\end{Theorem}

We provide a few easy examples to illustrate the result; Example (iii) shows
that if {\Ref{2014110901} is} not satisfied then in general we can't say the limit is Poisson.

\begin{Example}
\mbox{}
\begin{itemize}
\item[(i)] If $\Xi=\xi$ is non-random, $\Var({|\Xi|}(r))=0$ and Theorem~\ref{thm2} reiterates Theorem~\ref{thmshadow} that the limit process is Poisson.
\item[(ii)] If $\Xi$ is a Poisson process with intensity $\lambda$, then $\var ({|\Xi|}(r))/(\mean {|\Xi|}(r))^2=(\lambda\pi r^2)^{-1}$ which tends to zero
as $r\to\infty$, so the limit process is Poisson.
\item[(iii)] If $\Xi$ is a Cox process having intensity $\lambda_i{>0}$ with probability $1/2$ for $i=1,2$ and $\lambda_1\not=\lambda_2$, 
then it is clear from Theorem~\ref{thmPPP} 
that $N$ 
is a Cox process directed by a random mean measure that takes two measures with equal probability. Using the generic formula 
$\Var(X)=\mean\Var[X | \ca]+\Var( \E [X | \ca] )$, valid for any random variable $X$ and sigma-algebra $\ca$, we can derive that 
\[
\mean {|\Xi|}(r)=\frac{\lambda_1+\lambda_2}2\pi r^2,\,\,\,\,\,\,\,\, \var({|\Xi|}(r))=\frac{\lambda_1+\lambda_2}2\pi r^2+\frac{(\lambda_1-\lambda_2)^2}4\pi^2 r^4,
\]
and hence {$\lim_{r\downarrow 0}|\Lambda|(r)=0$, $\lim_{r\to\infty}|\Lambda|(r)=\infty$ but}
$$\lim_{r\to\infty}\frac{\Var ({|\Xi|}(r))}{\left(\mean {|\Xi|}(r)\right)^2}=\frac{(\lambda_1-\lambda_2)^2}{(\lambda_1+\lambda_2)^2}\ne0.$$
\end{itemize}
\end{Example}

\section{Applications}\label{secex} 
We assume throughout this section that %$\cs=\IR^d$, 
$g$ is a function on $\IR^d$ such that {$g$ is positive on $\IR^d/\{0\}$ and} $g(x)=h(\abs{x})$ for $h$ a {left continuous and nondecreasing} function with inverse $h^{-1}$. %=\inf\{x: h(x)\geq y\}$.
We apply Theorem~\ref{thmshadow} to fading tending to zero (Section~\ref{sec:fad}),
Theorem~\ref{thmdtv2} to a composite model (Section~\ref{sec:compo}),
Theorem~\ref{thmPPP} to transmitters placed according to a Poisson process (Section~\ref{expoisson}),
and Theorem~\ref{thm2} to transmitters placed according to an $\alpha$-Ginibre process (Section~\ref{sec:gini} where $\alpha$-Ginibre processes are defined).

\subsection{Fading tending to zero}\label{sec:fad}

As already discussed in Remark~\ref{rem:logno}, Theorem~\ref{thmshadow} can be applied to many standard models. 
We first demand that $\xi$ is such that there is a nondecreasing function $D$ satisfying 
\[
\lim_{r\to\infty} \frac{{|\xi|}(r)}{D(r)}=1 \,\,\,\,\mbox{ and } \,\,\,\,\, \lim_{r\to0} D(r)=0.
\]
This condition is satisfied for lattices on $\IR^d$, stationary and ergodic processes. In particular, specializing to the case $d=2$, if transmitters are
at the vertices of 
a
\begin{itemize}
\item triangular lattice, edge lengths $s$, excluding the origin, then $D(r)=2 \pi r^2/(\sqrt{3}s^2)$; 
\item hexagonal lattice, edge lengths $s$, excluding the origin, then  $D(r)= 4\pi r^2/(3\sqrt{3} s^2)$;
\item  square lattice, edge lengths $s$, excluding the origin, then $D(r) =\pi r^2/s^2$.
\end{itemize}
From this point, for a given $h$, we only need to have $S(\sigma)\convp0$
and $\mean D( h^{-1}(S(\sigma) t)) \to {L}(t)$ {for $t\in C(L)$}. Computing this expectation is straightforward under nice distributions of $S(\sigma)$;
even for composite models, {for example, when $S(\sigma)$} is a product of independent variables. 
Finally, it's worth repeating  that if ${L}$ equals zero or infinity, the theorem is still true (with the obvious interpretation of a Poisson with mean {zero or infinity}) though not {of practical interest}.

\subsection{Dependent composite fading}\label{sec:compo}
In Remark~\ref{rem:logno}, we showed that if the transmitters are placed on $\IR^2$, 
$g(x)=(K\abs{x})^\beta$ for $\beta>2, K>0$, $\xi$ is such that ${|\xi|(r)}/r^2\to \lambda \pi$,
%that there is a $D$ 
%satisfying the hypotheses of Theorem~\ref{thmshadow}:
%\[
%\lim_{r\to\infty} \frac{\xi(r)}{D(r)}=1 \,\,\,\,\mbox{ and } \,\,\,\,\, \lim_{r\to0} D(r)=0,
%\] 
and the composite fading distribution is 
\[
S(\sigma)=\exp\{\sigma B-\sigma^2/\beta\} S_F
\]
where $B$ is a standard normal {random variable} independent of $S_F$, and $S_F$ is exponential with rate $\Gamma(1+2/\beta)^{\beta/2}$ (or any other random variable such that $\mean S_F^{2/\beta}=1$ ),
then $N\t \sigma$ converges to Poisson process with mean measure
 \[
 M(t)=\frac{\lambda \pi t^{2/\beta}}{K^2}.
 \]
 (The rate of the exponential is only a scaling factor and convergence only 
 demands that $\mean S_F^{2/\beta}$ be finite, though this changes the mean measure.)
Now assume instead that each $S_i(\sigma)$, $i=1,2,\ldots$
is distributed as $S(\sigma)$, but rather than being i.i.d.\ they share a common 
$S_F$ variable which is not necessarily exponential. That is, define
\[
S_i=\exp\{\sigma B_i-\sigma^2/\beta\} S_F
\]
where $B_1,B_2,\ldots$ are i.i.d.\ standard normal random variables. Then conditional on $S_F$, $N\t \sigma$ converges 
to a Poisson process with mean measure 
\[
M^{S_F}(t){:}=\frac{S_F^{2/\beta}\lambda \pi t^{2/\beta}}{K^2}.
\]
Thus $N\t \sigma$ converges to a Cox process directed by $M^{S_F}$. 

%\cyan{Of couse, we could have chosen $S_F$ to be i.i.d. and $S_L$ to be a common variable (), and still have obtained a Cox process.  }

%Alternatively, we could have chosen $S_F$ to be i.i.d. and $S_L$ to be a common variable, and still have obtained a Cox process.  In other words, if one of the 
%{$S_L$ and $S_F$} in the product-based composite  model is common to all signals and the other is i.i.d. across each transmitter, then the resulting propagation process is again a Cox process, and {may not be} a Poisson process.

\subsection{Transmitters placed according to {a} Poisson process}\label{expoisson}   

As mentioned in the introduction, a common assumption is that the locations of transmitters $\Xi$ follow a Poisson process.
If $\Xi$ is a Poisson process then Theorem~\ref{thmPPP} shows that $N$ is a Poisson process. 
%;this generalizes several results in the literature; see \cite[Lemma~1]{blaszczyszyn2013using} and references there. 
Moreover, the mean measure of $N$ is computed {through \Ref{meanform2b}.}
%by first computing the conditional expectation given $\Xi$ via~\eq{meanform2} and then taking expectations to find that
%\ben{ \label{PPPmean}
%M(t)= \mean {|\Lambda|} ( h^{-1}(St)).
%}
We compute the mean measure in some examples.

\medskip
\noindent\textbf{Example A.} If $\Xi$ is a homogeneous Poisson process with intensity $\lambda$,
$g(x)=(K\abs{x})^\beta$,  for some $\beta>0, K>0$, then it's well known (see
\cite[Section I.A]{blaszczyszyn2014studying}) that $N$ is a Poisson process with  intensity measure depending on $S$ only through~$\mean S^{2/\beta}$, which has been referred to as \emph{propagation invariance}.% \cyan{and the corresponding signal-to-interference ratio is trivially related to a special case of the two-parameter Poisson-Dirichlet process~\cite{keeler2014sinr}}. 
So $h(r)= K^\beta r^\beta$ and  the mean measure as given by~\eq{meanform2} is
 \ba{
 M(t)	&=\lambda \pi \mean\left [h^{-1}(St)^2 \right]={\lambda \pi t^{2/\beta} \mean\left [S^{2/\beta} \right]}/K^2.
 }
  
  \medskip
\noindent\textbf{Example B.} More generally, if $h(r)= r^\beta e^{\alpha r}$ (used in the empirical work \cite{franceschetti2004random}) and
 $W$ is the Lambert W-function, i.e., $x=W(y)$ is the solution of $xe^x=y$, then {$h(r)=x$} gives $r{=h^{-1}(x):}=\frac \beta \alpha W(\alpha x^{1/\beta}/\beta)$, so Proposition~\ref{propmean2} implies
\ba{
M(t)&=\pi\lambda \mean \left(\frac \beta\alpha W\left(\frac \alpha \beta (t S)^{1/\beta}\right)\right)^2 \\
	&=\pi\lambda \mean\left[ (t S)^{2/\beta}e^{-2W\left(\frac \alpha \beta(t S)^{1/\beta}\right)} \right],
} 
where the second equality uses the fact that $W(y)^2=y^2e^{-2W(y)}$.

\medskip
\noindent\textbf{Example C.} Generalizations of the power-law path loss function are so-called multi-slope models which
have 
\be{
h(r)=\left(\sum_{i=1}^{{k+1}}  {\1_{r_{i-1}\leq r <r_i}}  b_i r^{-\beta_i}\right)^{-1}=\sum_{i=1}^{{k+1}}  {\1_{r_{i-1}\leq r <r_i}}  b_i^{{-1}} r^{\beta_i},
}
where {$\1$ is the indicator function}, $0={r_0<}r_1<\cdots< r_k<r_{k+1}=\infty$, $\beta_i>0$, and $b_i>0$ are chosen to make $h$ continuous; see \cite{dualslope} and references there. Since each interval $[r_{i-1},r_i)$ is disjoint with all others, the inverse of the multi-slope model is simply
\be{
h^{-1}(s)=\sum_{i=1}^{{k+1}}  {\1_{s_{i-1}\leq s   <s_i}}  c_i s^{1/\beta_i},
}
where $s_i=b_i^{{-1}} r_i^{\beta_i}$ {and $c_i=b_i^{1/\beta_i}$.} 
Theorem~\ref{thmPPP} says that the {propagation} process is Poisson and according to  expression~{\eq{meanform2b}}, the mean {measure} is
\begin{align}
M(t)&=2\pi\lambda \mean(h^{-1}(tS)^2)=2\pi\lambda \sum_{i=1}^k t^{2/\beta_i} c_i \mean\left[ {\1_{s_{i-1}\leq tS <s_i}}   S^{2/\beta_i} \right]. \label{20a}
\end{align}
%Theorem~\ref{thmPPP}  applies here and the formula~{\eq{meanform2b}} also holds and is easily computed in concrete cases. 
Note that in contrast to the propagation invariance of the case of the power-law path loss function of Example~B above,
where all $S$ with the same $2/\beta$ moment 
induce a common propagation process distribution, the form of the mean measure~\eq{20a}
suggests that no analogous simple invariance property holds for the multi-slope model.

\subsection{Transmitters placed as an $\alpha$-Ginibre process}\label{sec:gini}

The assumption that transmitters are placed according to a Poisson process can be heuristically justified by considering an ``average" observer with 
fixed transmitters
and is a convenient assumption due to its tractability. There has been recent interest in modeling the transmitter locations
according to other processes, especially those that exhibit repulsion or clustering among the points (some networks are designed 
to resemble lattices while others have clustering due to physical and technological considerations) \cite{soellerhaus2014} . One such process that exhibits repulsion is the $\alpha$-Ginibre process on the complex plane $\C$ which has been used to model mobile networks \cite{Miyoshi2014}, \cite{miyoshi2014cellular}.
The process is defined through the factorial moment measures: for a locally finite process $\Xi$ on a Polish space $\cs$, the $n$\,th order {\em factorial moment measure} $\nu^{(n)}$ of $\Xi$  is defined by the relation \cite[pp.~109--110]{Kallenberg1983} (also see~\cite[Chapter 9]{baccelli2009stochastic1})
\begin{eqnarray}
&& \E \left[ \int_{\cs^n}f(x_1,\dots,x_n) \Xi(dx_1)\left(\Xi-\delta_{x_1}\right)(dx_2)\dots \left(\Xi-\sum_{i=1}^{n-1}\delta_{x_i}\right)(dx_n)  \right]\nonumber\\
&&=\int_{\cs^n}f(x_1,\dots,x_n)  \nu^{(n)}(dx_1,\dots,dx_n),\label{ax14-09-26-1}
\end{eqnarray}
where $f$ ranges over all Borel measurable functions $h:\cs^n\rightarrow [0,\infty)$. 
The special case $\nu^{(1)} $ is simply the mean measure of $\Xi$. {To define the Ginibre process, for $x\in\C$, let $\bar x$ and $|x|$ be the complex conjugate and modulus of $x$.}

\begin{Definition}
We say the process {$\Xi$} on the complex plane {$\C$} is an $\alpha$-Ginibre process if its factorial moment measures are given by
\[
\nu\t n (dx_1, \ldots, dx_n) = {\rho\t n(x_1,\ldots, x_n) dx_1\dots dx_n},\ n\ge 1,
\]
{where $\rho\t n(x_1,\ldots, x_n)$ is the determinant of the $n\times n$ matrix with $(i,j)$th entry 
\[K_{\alpha,c}(x_i,x_j)=\frac c\pi e^{-\frac{c}{2\alpha}(|x_i|^2+|x_j|^2)}e^{\frac c\alpha x_i\bar x_j},\ c>0. \]}
\end{Definition}
In particular, direct computation gives
\[
\rho^{(1)}(x)=\frac c\pi>0 \,\,\,\,\, \mbox{ and } \,\,\,\,\, \rho^{(2)}(x,y)=\frac{c^2}{\pi^2}(1-e^{-\frac c\alpha|x-y|^2}).
\]

\begin{Theorem} \label{thm3} 
Assume that $\Xi$ is an  $\alpha$-Ginibre process on $\C$,
$(S(\sigma))_{\sigma\geq0}$ is a family of positive random variables, $g$ is a function such that {$g$ is positive on $\IR^d/\{0\}$ and}  $g(x)=h(\abs{x})$, with $h$ a {left continuous and nondecreasing}
function with inverse $h^{-1}$, and $N\t \sigma$ is the propagation process generated by $S(\sigma)$, $g$ and $\Xi$. If
 \be{
(i)\  S(\sigma) \convp 0  \,\,\,\,\,\, \mbox{ and } \,\,\,\,\,\,  (ii)\  \lambda \pi \mean \left[h^{-1}(S(\sigma) t))^2\right] \to {L(t)} {\mbox{ for all }t\in C(L)}
}
as $\sigma\to\infty$, then 
$N\t \sigma$ converges weakly to a Poisson process {$Z^L$} with mean measure $L$.
\end{Theorem}

\begin{proof}[Proof of Theorem~\ref{thm3}] 
We show $\Var({|\Xi|}(r))/(\mean{|\Xi|}(r))^2\to0$ as $r\to\infty$ and then the result follows by Theorem~\ref{thm2}.
Now, let $C_r=\{x\in{\C}:\ |x|\le r\}$, we find
\begin{eqnarray*}
\mean[{|\Xi|}(r)^2]&=&\mean\int_{C_r}\int_{C_r}\Xi(dx)(\Xi-\delta_x)(dy)+\mean\int_{C_r}\Xi(dx)\\
&=&\int_{C_r}\int_{C_r}\frac{c^2}{\pi^2}\left(1-e^{-\frac c\alpha|x-y|^2}\right)dxdy+\int_{C_r}\frac c\pi dx\\
&\le& c^2r^4+cr^2.
\end{eqnarray*}
Since $\mean {|\Xi|}(r)=cr^2$, {we find} that $\var({|\Xi|}(r))\le c r^2$ and $\var({|\Xi|}(r))/(\mean{|\Xi|}(r))^2\le 1/(cr^2)\to0$ as $r\to\infty.$
\end{proof}

For example, choosing $g(x)=(K\abs{x})^\beta$ for $\beta>2, K>0$, 
and 
\[
S(\sigma)=\exp\{\sigma B-\sigma^2/\beta\} S_F
\]
where $B$ is standard normal {random variable}, independent of $S_F$ and which {satisfies} $\mean S_F^{2/\beta}=1$ 
(for example $S_F=1$ or is exponential with rate $\Gamma(1+2/\beta)^{\beta/2}$)
then $N\t \sigma$ converges to Poisson process  with mean measure
 \[
 M(t)=\frac{\lambda \pi t^{2/\beta}}{K^2}.
 \]

\section{Proofs}\label{secproofs}

\begin{proof}[Proof of Theorem~\ref{thmdtv}]
We first show the upper bounds. Let $Z_i$ be {independent} Poisson processes on $\pR$ with mean measures $p_i(t)$ and $N_i$ be the process placing a single point at $Y_i$. 
Notice that
\be{
\sum_{i\in\mcI^\xi}Z_i{\equald}Z, \,\,\,\, \mbox{ and } \,\,\,\, \sum_{i\in\mcI^\xi}N_i{\equald}N,
}
where $\equald$ means they are equal in distribution, and  by independence we can bound
\ben{\label{t1}
\dtv(\law(Z\res \tt), \law(N\res \tt))\leq \sum_{i\in\mcI^\xi} \dtv(\law(Z_i\res \tt), \law(N_i\res \tt)).
}
Straightforward considerations show that for both $Z_i\res \tt$ and ${N_i}\res \tt$, given there is {a single} point in the interval
$(0,\tt]$, it is distributed with density
\be{
\frac{p_i(dt)}{p_i(\tt)}, \,\,\, \, 0< t \le \tt.
}
An alternative expression for the total variation distance, using Monge-Kantorovich duality~\cite{Rachev1984}, is
\be{
\dtv(\nu_1, \nu_2)=\inf_{(\xi_1, \xi_2)} \pr(\xi_1 \not= \xi_2), 
}
where the infimum is over all couplings of $\nu_1, \nu_2$. So from the previous {consideration}, 
if $Z_i(\tt)=N_i(\tt)$ (as above, $Z_i(\tt)$ is the number of points of $Z_i$ in the interval $(0,\tt]$), then we can
couple $Z_i\res \tt$ and $N_i\res \tt$ exactly. Thus, we easily find
\ben{\label{t2}
\dtv(\law(Z_i\res \tt), \law(N_i\res \tt)){=} \dtv(\law(Z_i(\tt)), \law(N_i(\tt)))\leq p_i(\tt)^2;
}
the last inequality follows by noting that $Z_i(\tt)$ is Poisson with mean $p_i(\tt)$ and $N_i(\tt)$ is a Bernoulli with success probability
$p_i(\tt)$ and then using well-known bounds between Poisson {random} variables and Bernoullis~\cite{LeCam1960}; see also \cite[Formula~(1.8)]{Barbour1992}. 
Combining~\eq{t1} and~\eq{t2} proves the first upper bound of the theorem and the second is simple.

For the lower bound, note that the number of points of $N$, respectively $Z$, falling in $(0,\tt]$ is a measurable function of $N\res \tt$, respectively $Z\res \tt${,} so
\[
\dtv(\law(N(\tt)), \law(Z(\tt))) \leq \dtv(\law(N\res \tt), \law(Z\res \tt)).
\]
But as already observed, $N(\tt)$ is a sum of independent indicators and $Z(\tt)$ is a Poisson distribution having mean $M(\tt)$ common with $N(\tt)$. Thus the lower bound
of \cite[Theorem~2]{Barbour1984} applies which is exactly the lower bound of the theorem.
\end{proof}

We use Theorem~\ref{thmdtv} to prove convergence results through the following lemma which is easily
proved from \cite[Theorem~4.2]{Kallenberg1983}.
\begin{Lemma}\label{lemPPPcon}
If $N\t n$ is a sequence of processes on $\pR$ such that there is a set $\{t_i:i\in\IN\}$ with  $0<t_i\uparrow\infty$ and for each $t_i$,
$N\t n\res{t_i} \convd  Z \res{t_i}$ as $n\to\infty$ in the vague topology for some process $Z$ on $\pR$, then
$N\t n\convd Z$ as $n\to\infty$.
\end{Lemma}

\begin{proof} [Proof of Corollary~\ref{corollary1}] (c.f., \cite{Kallenberg1975, Brown1979, Schuhmacher2005, Schuhmacher2009}) For each $t>0$, let $\chi_t$ (resp. $\ch_t\subset \chi_t$) be the set of all finite Radon measures (resp. point measures) on $(0,t]$
and $\ck_t$ be the set of all Lipschitz functions on $(0,t]$ with respect to the metric $d_0(x,y)=\min\{1,|x-y|\}$, i.e., $\ck_t=\{k:\ |k(x)-k(y)|\le d_0(x,y),\ x,y\in(0,t]\}$. \cite{BarbourBrown1992} introduce 
a Wasserstein metric $d_{1t}$ for finite measures $\eta_1,\eta_2\in\chi_t$ as
$$d_{1t}(\eta_1,\eta_2)=\left\{\begin{array}{ll}
1&\mbox{ if }{\eta_1(t)}\ne {\eta_2(t)},\\
0&\mbox{ if }{\eta_1(t)}={\eta_2(t)}=0,\\
\frac{\sup_{k\in\ck_t}\left\vert \int_{(0,t]}k(x)\eta_1(dx)-\int_{(0,t]}k(x)\eta_2(dx)\right\vert}{{\eta_1(t)}}&\mbox{ if }{\eta_1(t)}={\eta_2(t)}>0.
\end{array}\right.$$
Moreover, if $\eta_1=\sum_{j=1}^m\delta_{x_j}$ and $\eta_2=\sum_{j=1}^m\delta_{y_j}$ with $m>0$ and $\{x_j,y_j:\ 1\le j\le m\}\subset (0,t]$, one can write 
$$d_{1t}(\eta_1,\eta_2)=\min_\pi\left\{m^{-1}\sum_{j=1}^md_0(x_j,y_{\pi(j)})\right\},$$
where $\pi$ ranges over all permutations of $\{1,\dots,m\}$ \cite[Section~2.2]{Rachev1984}. The metric $d_{1t}$ quantifies the vague (and the weak) topology on
$\ch_t$ \cite{Xia2005}. \cite{BarbourBrown1992} then introduce a Wasserstein metric 
$d_{2t}$ induced by $d_{1t}$ for two distributions $Q_1$ and $Q_2$ on $\ch_t$ as
$$d_{2t}(Q_1,Q_2)=\sup_f\left\vert \int_{\ch_t}fdQ_1-\int_{\ch_t}fdQ_2\right\vert,$$
where the supremum is taken over all $d_{1t}$-Lipschitz functions on $\ch_t$. Let $Z\t n$ be a Poisson process on ${\pR}$ with mean measure {$M\t n(\cdot):=\mean N^{(n)}(\cdot)$}. 
 Now applying the triangle inequality, using the fact that $d_{2t}(Q_1, Q_2)\leq \dtv(Q_1, Q_2)$ and then applying the results \cite[Theorem~1.5]{BrownXia1995}
 and Theorem~\ref{thmdtv}, we obtain for $t\in C(L)$,
\begin{eqnarray*}
&&d_{2t}(\law(N\t n \res t), \law(Z \res t))\\
&&\le d_{2t}(\law(N\t n \res t), \law(Z\t n\res t))+d_{2t}(\law(Z\t n \res t), \law(Z \res t))\\
&&\le \dtv(\law(N\t n \res t), \law(Z\t n\res t))+d_{1t}(M\t n\res t/M\t n(t),L\res t/L(t))+|M\t n(t)-L(t)|\\
&&\le\sum_{i\in\mcI_n} \pr(Y_{ni}\leq t)^2+d_{1t}(M\t n\res t/M\t n(t),L\res t/L(t))+|M\t n(t)-L(t)|\to 0
\end{eqnarray*} as $n\to\infty$.
Thus  $N\t n \res t\convd  Z \res t$ as $n\to\infty$ for all $t\in C(L)$
and the claim follows from Lemma~\ref{lemPPPcon}. %\cite[Theorem~4.2]{Kallenberg1983}.
\end{proof}

\begin{proof}[Proof of Proposition~\ref{propmean}]
Both assertions follow from writing 
\be{
{|\xi|}(r)=\sum_{i\in \mcI^{{\xi}}} {\1_{{0<}\abs{x_i}\leq r}}.
}
The representation implies 
\be{
\int_{0}^\infty \pr\left({0<}\frac{h(r)}{S}\leq t\right) {|\xi|}(dr)=\sum_{i\in\mcI^{{\xi}}} \pr\left({0<}\frac{g(x_i)}{S}\leq t\right) =M(t).
}
For the second assertion, the indicator representation and the fact that 
\[
h(x)\leq y \iff x\leq h^{-1}(y)
\]
implies
\be{
\mean \left[{|\xi|}(h^{-1}(St))\right]=\mean \sum_{i\in\mcI^{{\xi}}} {\1_{{0<}\abs{x_i}\leq h^{-1}(St) }}=\mean \sum_{i\in\mcI^{{\xi}}} {\1_{{0<}g(x_i)\leq St}}=M(t).
}
\end{proof}

\begin{proof}[Proof of Theorem~\ref{thmPPP}] 
That $Z$ is a Poisson process follows since $\{(x_i, S_i): x_i\in\Theta\}$ is a Poisson point
process on $\IR^d\times \pR$ and the points of $Z$ are a measurable function of this
process. The computation of the mean measure is straightforward and the bound on the total variation distance is easy to see from the coupling definition of the total variation distance {since}
 one can construct a coupling in such a way that if $\Theta=\Xi$ then $Z=N$.
\end{proof}

\begin{proof}[Proof of Theorem~\ref{thmdtv2}] 
We use the following inequality relating the total variation distance of conditioned {random} variables to the unconditional; see, for example, \cite[Section~3]{Roellin2014},
\be{
\dtv(\law(X), \law(Y))\leq \mean \dtv(\law(X|W), \law(Y|W)).
}
We use this inequality to find
\ben{\label{mmm}
\dtv(\law(Z\res \tt), \law(N\res \tt))\leq \mean \dtv(\law(Z^\Xi \res \tt), \law(N^\Xi \res \tt)),
}
where $N^\Xi$ and $Z^\Xi$ denote the processes $N$ and $Z$ conditional
on $\Xi$.
Since $Z^\Xi$ is a Poisson process with mean measure $M^\Xi$, we apply Theorem~\ref{thmdtv} to obtain
\be{
\dtv(\law(Z^\Xi \res \tt), \law(N^\Xi \res \tt))\leq  {\int_{\IR^d}p_i^{(x)}(\tt)^2\Xi(dx)}
}
and now taking expectations and using~\eq{mmm} implies the theorem.
\end{proof}

\begin{proof}[Proof of Theorem~\ref{thmshadow}]
%The first assertion is an easy consequence of Theorem~\ref{thmdtv} and the alternate expression of the mean measure~\eq{meanform2}.
The theorem follows easily from Corollary~\ref{corollary1} and \Ref{meanform2} once we establish that, {for $t\in C(L)$,}
\ben{\label{22}
\lim_{\sigma\to\infty} \mean D( h^{-1}(S(\sigma) t))=\lim_{\sigma\to\infty} \mean {|\xi|}( h^{-1}(S(\sigma) t)).
}
To show~\eq{22}, let $\eps>0$ and $r_{\eps}$ be such that for $r{\ge}r_{\eps}$
\ben{\label{24}
1-\eps < \frac{{|\xi|}(r)}{D(r)} < 1+\eps.
}
Then denoting the distribution function of $S(\sigma)$ by $F_\sigma$, we have
\ban{
&\limsup_{\sigma\to\infty}  \mean {|\xi|}( h^{-1}(S(\sigma) t)) \notag\\
&\qquad\leq \limsup_{\sigma\to\infty}  \int_0^{h(r_\eps)/t}  {|\xi|}( h^{-1}(s t)) F_\sigma(ds)+\limsup_{\sigma\to\infty}\int_{h(r_\eps)/t}^\infty  {|\xi|}( h^{-1}(s t)) F_\sigma(ds) \notag\\
&\qquad = \limsup_{\sigma\to\infty}\int_{h(r_\eps)/t}^\infty  {|\xi|}( h^{-1}(s t)) F_\sigma(ds); \label{23}
}
the equality is because $S(\sigma)\convp 0$ and $\lim_{y\to 0} h^{-1}(y)=0$. Now using the definition~\eq{24} of $r_\eps$ and noting {that 
$s\ge h(r_\eps)/t$ implies $h^{-1}(st)\ge r_\eps$}, 
we bound~\eq{23} from above to find
\ba{
\limsup_{\sigma\to\infty}  \mean {|\xi|}( h^{-1}(S(\sigma) t))& \leq (1+\eps) \limsup_{\sigma\to\infty} \int_{h(r_\eps)/t}^\infty  D( h^{-1}(s t)) F_\sigma(ds) \\
	&\leq (1+\eps) \limsup_{\sigma\to\infty} \int_{0}^\infty  D( h^{-1}(s t)) F_\sigma(ds) \\
	&=(1+\eps) {\lim_{\sigma\to\infty}}\mean D( h^{-1}(S(\sigma) t)).
}
{As $\eps$ is arbitrary, this yields
\ban{
\limsup_{\sigma\to\infty}  \mean {|\xi|}( h^{-1}(S(\sigma) t))\le \lim_{\sigma\to\infty}\mean D( h^{-1}(S(\sigma) t)).\label{14102402}}}
Similarly, 
\ben{\label{25}
\liminf_{\sigma\to\infty}  \mean {|\xi|}( h^{-1}(S(\sigma) t))\geq (1-\eps) \liminf_{\sigma\to\infty} \int_{h(r_\eps)/t}^\infty  D( h^{-1}(s t)) F_\sigma(ds).
}
But again using that $S(\sigma)\to0$ in probability, $\lim_{y\to 0} h^{-1}(y)=0$, and now also that $\lim_{r\to0}D(r)\to 0$,
\be{
\lim_{\sigma\to\infty}  \int_{h(r_\eps)/t}^\infty D( h^{-1}(s t)) F_\sigma(ds)=\lim_{\sigma\to\infty} \int_{0}^\infty  D( h^{-1}(s t)) F_\sigma(ds),
}
and combining this with~\eq{25} implies
\be{
\liminf_{\sigma\to\infty}  \mean {|\xi|}( h^{-1}(S(\sigma) t))\geq (1-\eps) {\lim_{\sigma\to\infty}}\mean D( h^{-1}(S(\sigma) t)).
}
Since $\eps$ was arbitrary, we have that 
\be{
\liminf_{\sigma\to\infty}  \mean {|\xi|}( h^{-1}(S(\sigma) t)) \geq {\lim_{\sigma\to\infty} } \mean {D}( h^{-1}(S(\sigma) t))
}
which, {together with \Ref{14102402},} proves~\eq{22}.
\end{proof}

\begin{proof}[Proof of Theorem~\ref{thm2}] Write $\pxs(t)=\pr({0<}g(x)/S(\sigma)\le t)$,
$\MXs(t)=\int_{\IR^d} \pxs(t)\Xi(dx)$ and $\MLs(t)=\int_{\IR^d} \pxs(t)\Lambda(dx)$. We divide the proof into three steps. 

\noindent(i) $\MXs\convd L$ in the vague topology as $\sigma\to\infty$.

To show the claim, from \cite[Theorem~4.2]{Kallenberg1983}, it suffices to show that
for each continuous function $f:\pR\to\IR_+$ with compact support and bounded continuous first derivative,
\ben{\int_{\pR}f(t)d\MXs(t)\convd \int_{\pR}f(t)dL(t).\label{14102601}}
Now, let $F_\sigma$ be the distribution function of $S(\sigma)$ and let $0<a<b<\infty$ such that $b\in C(L)$ and the support of $f$ is contained in $[a,b]$
{such} that 
$f(b)=f(a)=0$. 
It follows from \Ref{meanform2} that
\[\MXs(t)=\int_{\pR}{|\Xi|}(h^{-1}(st))dF_\sigma(s)\]
and, by taking expectation,
\[\MLs(t)=\int_{\pR}{|\Lambda|}(h^{-1}(st))dF_\sigma(s).\]
Using Fubini's Theorem and noting $f(b)=f(a)=0$, we have 
\ben{\label{20001}
\int_{\pR}f(t)d\MXs(t)=\int_{\pR}\int_a^t f'(s) ds d\MXs(t)=-\int_a^b \MXs(t) f'(t) dt
}
and hence
\ba{
\int_{\pR}f(t)d\MXs(t) %&=\int_0^\infty\left(\int_a^bf(t)d\Xi(h^{-1}(st))\right)dF_\sigma(s)\\
&=-\int_0^\infty\int_a^b{|\Xi|}(h^{-1}(st))f'(t)dtdF_\sigma(s)\\
&=-\int_{\theta_0}^\infty\int_a^b[{|\Xi|}(h^{-1}(st))-{|\Lambda|}(h^{-1}(st))]f'(t)dtdF_\sigma(s)\\
&\ \ \ \ -\int_0^\infty\int_a^b{|\Lambda|}(h^{-1}(st))f'(t)dtdF_\sigma(s)\\
&\ \ \ \ +\int_0^{\theta_0}\int_a^b{|\Lambda|}(h^{-1}(st))f'(t)dtdF_\sigma(s)\\
&\ \ \ \ -\int_0^{\theta_0}\int_a^b{|\Xi|}(h^{-1}(st))f'(t)dtdF_\sigma(s)\\
&=:\mbox{(I)+(II)+(III)+(IV)},}
where $\theta_0$ is chosen such that 
\[\frac{\var\left({|\Xi|}(h^{-1}(sa))\right)}{{|\Lambda|}(h^{-1}(sa))^2}\le 1,\ \forall s\ge \theta_0.\]
We complete the proof of \Ref{14102601} by showing that, as $\sigma\to\infty$, (a) (I)$\convp 0$; (b) (II)$\to\int_{\pR}f(t)dL(t)$; (c) (III)$\to 0$ and (d) (IV)$\convp0$.

(a) Let $v_\sigma$ be the variance of (I), write $\Xi'(t):={|\Xi|}(h^{-1}(t))$ and 
$\Lambda'(t):={|\Lambda|}(h^{-1}(t))$, then 
using the geometric-arithmetic mean inequality $AB\leq (A^2+B^2)/2$ and symmetry to obtain the first inequality, we have
\ban{v_\sigma&=\int_{\theta_0}^\infty\int_a^b\int_{\theta_0}^\infty\int_a^b
\mean \left(\frac{\Xi'(s_1t_1)}{\Lambda'(s_1t_1)}-1\right)f'(t_1)\left(\frac{\Xi'(s_2t_2)}{\Lambda'(s_2t_2)}-1\right)f'(t_2)\nonumber\\
&\mbox{\hskip5.5cm} \times \Lambda'(s_1t_1)dt_1dF_\sigma(s_1)
\Lambda'(s_2t_2)dt_2dF_\sigma(s_2)\nonumber\\
&\le\int_{\theta_0}^\infty\int_a^b\int_{\theta_0}^\infty\int_a^b
\frac{\var(\Xi'(s_1t_1))}{\Lambda'(s_1t_1)^2}f'(t_1)^2\Lambda'(s_1t_1)dt_1dF_\sigma(s_1)
\Lambda'(s_2t_2)dt_2dF_\sigma(s_2)\nonumber\\
&\le\|f'\|^2(b-a)\mean\Lambda'(S(\sigma) b)\int_{\theta_0}^\infty\Lambda'(s_1b)\int_a^b
\frac{\var(\Xi'(s_1t_1))}{\Lambda'(s_1t_1)^2}dt_1dF_\sigma(s_1),\label{14102602}
} 
where $\|f'\|=\sup_{t\in\pR}|f'(t)|$.
For each $\epsilon>0$, let $T_\epsilon>\theta_0$ such that 
\[\frac{\var\left(\Xi'(sa)\right)}{\Lambda'(sa)^2}\le \epsilon,\ \forall s\ge T_\epsilon.\]
It follows from \Ref{14102602} that
\ban{v_\sigma&\le\|f'\|^2(b-a)\mean\Lambda'(S(\sigma) b)\left(\int_{\theta_0}^{T_\epsilon}\int_a^b
+\epsilon\int_{T_\epsilon}^\infty\int_a^b\right)
\Lambda'(s_1b)dt_1dF_\sigma(s_1)\nonumber\\
&\le\|f'\|^2(b-a)^2\mean\Lambda'(S(\sigma) b)\left(\mean \Lambda'(S(\sigma)b)\1_{S(\sigma)\le T_\epsilon}+\epsilon \mean\Lambda'(S(\sigma)b)\right).\label{14102603}
}
Using that $S(\sigma)\convp 0$, $\Lambda'(S(\sigma)b)\1_{S(\sigma)\le T_\epsilon}\le \Lambda'(T_\epsilon b)<\infty$, $\lim_{t\downarrow0}h^{-1}(t)=0$ and $\lim_{r\downarrow 0}{|\Lambda|}(r)=0$, we apply the bounded convergence theorem to obtain
\ben{\lim_{\sigma\to\infty}\mean \Lambda'(S(\sigma)b)\1_{S(\sigma)\le T_\epsilon}=0,\label{14102604}}
which, together with \Ref{14102603}, ensures
\be{\limsup_{\sigma\to\infty}v_\sigma\le\|f'\|^2(b-a)^2\epsilon L(b)^2.}
This yields $\lim_{\sigma\to\infty}v_\sigma=0$ due to the arbitrariness of $\epsilon$.

(b) Applying the dominated convergence theorem, we have
\ba{&-\int_0^\infty\int_a^b{|\Lambda|}(h^{-1}(st))f'(t)dtdF_\sigma(s)\\
&=-\int_a^b\mean{|\Lambda|}(h^{-1}(S(\sigma)t))f'(t)dt\\
&\to-\int_a^bf'(t)L(t)dt=\int_a^b f(t)dL(t)=\int_{\pR}f(t)dL(t), 
}
where the penultimate equality is due to Fubini's Theorem similar to~\eq{20001}.

(c) Use the same reasoning as that for \Ref{14102604}; as $\sigma\to\infty$,
\[\int_0^{\theta_0}\int_a^b{|\Lambda|}(h^{-1}(st))|f'(t)|dtdF_\sigma(s)
\le \|f'\|(b-a)\mean {|\Lambda|}(h^{-1}(S(\sigma)b))\1_{S(\sigma)\le \theta_0}\to 0.
\]

(d) It follows from (c) that, as $\sigma\to\infty$,
\[\mean\int_0^{\theta_0}\int_a^b{|\Xi|}(h^{-1}(st))|f'(t)|dtdF_\sigma(s)
\le \|f'\|(b-a)\mean{|\Lambda|}(h^{-1}(S(\sigma)b))\1_{S(\sigma)\le \theta_0}\to0.\]

{At this point the proof follows along the lines of Corollary~\ref{corollary1}.}

\noindent (ii) With the notations in the proof of Corollary~\ref{corollary1}, let {$Z^L$} be a Poisson process with mean measure $L$, then for each $t>0$,
\ban{&d_{2t}(\cl(N^{(\sigma)}|_t),\cl({Z^L}|_t))\nonumber\\
&\le \int_{\IR^d} \pxs(t)^2\Lambda(dx)+
\mean d_{1t}(\MXs|_t/\MXs(t),L|_t/L(t))+\mean 1\wedge|\MXs(t)-L(t)|.\label{14102605}
}

To show (ii), let $Z_\sigma^\Xi$ be a Cox process directed by the mean measure $\MXs$, then 
\[d_{2t}(\cl(N^{(\sigma)}|_t),\cl(Z_\sigma^\Xi|_t))\le d_{TV}(\cl(N^{(\sigma)}|_t),\cl(Z_\sigma^\Xi|_t))\]
and the first term in the upper bound of \Ref{14102605} comes from Theorem~\ref{thmdtv2}. For a Poisson process $\tilde Z$ on $\pR$ with mean measure $\lambda$, \cite[Theorem~1.5]{BrownXia1995} gives
\[d_{2t}(\cl(\tilde Z|_t),\cl({Z^L}|_t))\le d_{1t}(\lambda|_t/\lambda(t),L|_t/L(t))+1\wedge|\lambda(t)-L(t)|,\]
where $1$ is because  $d_{2t}\le 1$. Hence,
\ba{
d_{2t}(\cl(Z_\sigma^\Xi|_t),\cl({Z^L}|_t))&=\sup_f|\mean f(Z_\sigma^\Xi|_t)-\mean f({Z^L}|_t)|\\
&\le \mean \sup_f|\mean [f(Z_\sigma^\Xi|_t)|\Xi]-\mean f({Z^L}|_t)|\\
&\le \mean d_{1t}(\MXs|_t/\MXs(t),L|_t/L(t))+\mean1\wedge|\MXs(t)-L(t)|,
}
where $f$ ranges over all $d_{1t}$-Lipschitz functions on $\ch_t$. This gives the remaining two terms in the upper bound of \Ref{14102605}.

\noindent (iii) $N^{(\sigma)}$ converges weakly to a Poisson process with mean measure $L$.

To prove this claim, Lemma~\ref{lemPPPcon} implies it suffices to show that for each fixed $t\in C(L)$, the upper bound of \Ref{14102605} converges to $0$ as $\sigma\to\infty$. 
For the first term, we have, for each $r>0$,
\[\int_{\IR^d}\pxs(t)^2\Lambda(dx)\le {|\Lambda|}(r)+\pr(S(\sigma)\ge h(r)/t)\mean{|\Lambda|}(h^{-1}(S(\sigma)t)),\]
which, together with the assumption $S(\sigma)\convp 0$, implies
\[\limsup_{\sigma\to\infty}\int_{\IR^d}\pxs(t)^2\Lambda(dx)\le{|\Lambda|}(r).\]
Since $r$ is arbitrary and $\lim_{r\to0}{|\Lambda|}(r)=0$, we obtain
$\lim_{\sigma\to\infty}\int_{\IR^d}\pxs(t)^2\Lambda(dx)=0.$
For the second term in the upper bound of \Ref{14102605}, as $d_{1t}(\cdot,L|_t/L(t))$ is a continous function on 
\[\{v:\ v\mbox{ is {right continuous and nondecreasing} on }(0,t]\mbox{ with }
\lim_{s\downarrow 0}v(s)=0,\ v(t)=1\}\]
that equals zero at $L|_t/L(t)$, it follows from (i) and the continuity theorem that, as $\sigma\to\infty$,
\[d_{1t}(\MXs|_t/\MXs(t),L|_t/L(t))\convd0\]
 and hence
 $\mean d_{1t}(\MXs|_t/\MXs(t),L|_t/L(t))\to 0$. The third term converges to $0$ due to (i) and the bounded convergence theorem.

\end{proof}

\section{Final remarks}\label{14102301}
We show that under general conditions the wireless network signals appear to different observers at different times as different Poisson processes or different realizations of the same Cox process. This line of work strongly suggests that given the network is sufficiently large and stationary (or just isotropic) with strong enough random propagation effects such as fading and shadowing, then 
the signal strengths can be modeled directly as a Poisson or Cox process on the real line and the details of the distribution of the positioning of transmitters on the plane can be safely ignored. From the results presented here, there are 
many further directions of study:
For a given transmitter configuration, do some fading models induce a propagation process significantly closer to Poisson than others?
How do our results translate to functions of the propagation process, for example, the signal-to-interference ratio discussed in the introduction (c.f., \cite{keeler2014sinr})?
Can our results be extended to models with short range (spatial) dependence between the fading variables?

%%%%%%%%%%%%%%%%%%%%%%%%%%%%%
%%%%%%% References   %%%%%%%%
%%%%%%%%%%%%%%%%%%%%%%%%%%%%%

\def\ac{{Academic Press}~}
\def\aap{{Adv. Appl. Prob.}~}
\def\ap{{Ann. Probab.}~}
\def\anap{{Ann. Appl. Probab.}~}
\def\jap{{J. Appl. Probab.}~}
\def\jws{{John Wiley $\&$ Sons}~}
\def\ny{{New York}~}
\def\ptrf{{Probab. Theory Related Fields}~}
\def\sp{{Springer}~}
\def\spa{{Stochastic Processes and their Applications}~}
\def\sv{{Springer-Verlag}~}
\def\tpa{{Theory Probab. Appl.}~}
\def\zw{{Z. Wahrsch. Verw. Gebiete}~}

%%%%%%%%%%%%%%%%%%%%%%%%%%%%%%

\end{document}